\documentclass[reqno,centertags,a4paper,12pt]{amsart}

\usepackage{amsfonts, amsmath, amssymb, amsthm}
\usepackage{latexsym}
\usepackage{eucal}
\usepackage[dvips]{graphics}
\usepackage[english]{babel}
\usepackage[latin1]{inputenc}
\usepackage{enumerate}

\numberwithin{equation}{section}

\newtheorem{Theorem}{Theorem}[section]
\newtheorem{Lemma}{Lemma}[section]

\newcommand{\pd}{\partial}

\def\R{\mathbb{R}}

\def\Z{\mathbb{Z}}
\def\C{\mathbb{C}}

\begin{document}
\title[Bousinesq type system]{On the Cauchy problem for a Boussinesq type system}

\author[F. Linares]{F.~Linares}
\author[M.~Panthee]{M.~Panthee}
\author[J.~Drumond~Silva]{J.~Drumond~Silva}

\address{Felipe~Linares: IMPA, Estrada Dona Castorina 110, 22460--320, Rio de Janeiro, RJ, Brazil.}
\email{linares@impa.br}
\address{Mahendra~Panthee:
 CMAT, University of Minho, Campus de Gualtar, 4710--057, Braga Portugal.}
\email{mpanthee@math.uminho.pt}
\address{Jorge~Drumond~Silva:
Center for Mathematical Analysis, Geometry and Dynamical Systems,
Departamento de Matem\'atica,
Instituto Superior T\'ecnico,
 Av. Rovisco Pais, 1049-001 Lisboa, Portugal.}
\email{jsilva@math.ist.utl.pt}



\thanks{Research partially supported by the FCT/CAPES ``Nonlinear Waves and Dispersion'' Project.
 J. Drumond Silva was partially supported by the Center for Mathematical Analysis,
Geometry and Dynamical Systems through the Funda\c{c}\~ao para a Ci\^encia e Tecnologia (FCT/Portugal)
Program POCTI/FEDER. M. Panthee was supported  by grant Est-C/MAT/UI0013/2011 from FEDER Funds through ``Programa Operacional Factores de Competitividade - COMPETE'', and by grant PTDC/MAT/109844/2009 from Portuguese Funds through FCT - ``Funda\c c\~ao para a Ci\^encia e a Tecnologia''.}

\begin{abstract} We consider the initial value problem (IVP) associated to a Boussinesq type system. After rewriting
the system in an equivalent form of coupled KdV-type equations, we prove that this is locally well-posed in $(H^s(\R^2))^4$, $s>3/2$, using sharp smoothing estimates. Consequently we obtain the local well-posedness result for the original system in $H^s\times \mathcal{V}^{s+1}$ for $s>3/2$ (see below for the definition of $\mathcal{V}^{s}$).
\end{abstract}

\maketitle


\section{Introduction}

We consider the following Boussinesq type system
\begin{equation}\label{ivp1}
\begin{cases}
\eta_t+\Delta \Phi -\frac\mu6\Delta^2\Phi =- \epsilon \nabla \cdot\big[ \eta\big((\pd_{x_1}\Phi)^p,(\pd_{x_2}\Phi)^p \big)
\big],\\
\Phi_t+\eta-\mu(\sigma -1/2)\Delta\eta = - \frac\epsilon{p+1}\big((\pd_{x_1}\Phi)^{p+1}+(\pd_{x_2}\Phi)^{p+1}\big),
\end{cases}
(t, x) =(t, x_1, x_2)\in \R^{1+2},
\end{equation}
where $\epsilon$ is the amplitude parameter, $\mu$ is the long-wave parameter and $\sigma$ is the Bond number.
This system was derived in \cite{Q2} as a model for the evolution of three dimensional long water waves with small amplitude, where $\Phi$ represents the non dimensional velocity potential at the bottom $z=0$ and $\eta$ represents the free surface elevation.

Well-posedness of the associated Cauchy problem and stability of solitary waves for this system was studied in \cite{Q1}. In this paper the author  proved that the initial value problem (IVP) associated to \eqref{ivp1} is locally well-posed for given initial data in $H^s(\R^2)\times \mathcal{V}^{s+1}(\R^2)$, $s>2$, where $H^s(\R^2)$ is the usual $L^2$-based Sobolev space and $\mathcal{V}^s(\R^2)$ is the Hilbert space defined by
$$\mathcal{V}^s(\R^2):=\big\{f\in \mathcal{S}'(\R^2): \sqrt{-\Delta}f \in H^{s-1}\big\},$$
with the corresponding norm $\|f\|_{\mathcal{V}^s(\R^2)}=\|\sqrt{-\Delta}f\|_{H^{s-1}(\R^2)}$.
The idea used in \cite{Q1} is based on a change of variables by taking $u=\pd_{x_1} \Phi$ and $v=\pd_{x_2} \Phi$, thus enlarging the system to a new one for $U=(\eta,u,v)$, and then employing Kato's approach for the generalized Korteweg-de Vries (KdV) equation. The author also proves the existence of global weak solutions in time, for small initial data, by considering a variational approach, which involves the Hamiltonian defined in the energy space $H^1\times \mathcal{V}^2$.

In this work we aim to improve the local results obtained in \cite{Q1}. For the sake of simplicity
we consider $\mu =6$, $\sigma=2/3$ and $\epsilon =1$ in \eqref{ivp1} and we start by diagonalizing the constant coefficient
linear part of
the resulting system
\begin{equation} \label{mainsystem}
\begin{cases}
\eta_t+\Delta \Phi -\Delta^2\Phi =- \nabla \cdot\big[ \eta(\Phi_{x_1}^p,\Phi_{x_2}^p )
\big],\\
\Phi_t+\eta-\Delta\eta = - \frac 1 {p+1} (\Phi_{x_1}^{p+1}+\Phi_{x_2}^{p+1}),
\end{cases}
\end{equation}
to obtain the following equivalent system
\begin{equation}\label{ivp2}
\begin{cases}
u_t+ i\sqrt{-\Delta}(1-\Delta)u = -N_1(u,v) -N_2(u,v),\\
v_t- i\sqrt{-\Delta}(1-\Delta)v = N_1(u,v) -N_2(u,v),
\end{cases}
\end{equation}
where the nonlinearities are given by
\begin{equation}\label{eq1.3}
\begin{cases}
N_1(u,v)=\frac1{2\sqrt{-\Delta}}\Big\{\partial_{x_1}\big[\{\sqrt{-\Delta}(u-v)\}(u_{x_1}+v_{x_1})^p\big]+\partial_{x_2}\big[\{\sqrt{-\Delta}(u-v)\}(u_{x_2}+v_{x_2})^p\big] \Big\},\\
N_2(u,v) = \frac1{2(p+1)}\big[(u_{x_1}+v_{x_1})^{p+1} + (u_{x_2}+v_{x_2})^{p+1}\big].
\end{cases}
\end{equation}

In the diagonalization process the relation between the original functions $(\eta, \Phi)$ and the
new functions $(u,v)$ is given by
\begin{equation}
\begin{cases}
u= -\frac{i}{2\sqrt{-\Delta}}\eta + \frac12 \Phi,\\
v = \frac{i}{2\sqrt{-\Delta}}\eta + \frac12 \Phi,
\end{cases}
\end{equation}
so that $(\eta, \Phi)$ can be recovered from $(u,v)$ by
\begin{equation}
\label{eq1.6}
\begin{cases}
\eta = i\sqrt{-\Delta}u-i\sqrt{-\Delta}v,\\
\Phi= u+v.
\end{cases}
\end{equation}

This change of variables transforms the original system \eqref{mainsystem}, with different differentiability requirements
for $(\eta, \Phi)$, into a more symmetric form \eqref{ivp2} of
two coupled KdV-type equations with the same third order dispersive operator in the spatial variables
$i\sqrt{-\Delta}(1-\Delta)$ for $(u,v)$. The main difficulty now stems from the nonlinearities \eqref{eq1.3},
involving powers of the
derivatives of the unknowns. To deal with this issue, we differentiate the diagonalized system in $x_1$ and $x_2$,
enlarging it to a system of four equations for the partial derivatives of the unknowns $(u_{x_1},u_{x_2},v_{x_1},v_{x_2})$
(see equation \eqref{eq3.1} below) thus turning the nonlinearities into standard KdV-type ones. It is this final system
that we then solve, using classical methods (see \cite{KPV4}) involving Strichartz, maximal and smoothing estimates, which are developed in the next section of this
paper.

In the third section we then prove the local well-posedness result that constitutes the main theorem of this paper.
\begin{Theorem}\label{mainTh} Let $p=1$ and $s>3/2$, then for given data
$(\eta_0, \Phi_0)\in H^s(\R^2)\times \mathcal{V}^{s+1}(\R^2)$, there exist a time $T=T\big(\|\eta_0\|_{H^s(\R^2)}+\|\Phi_0\|_{\mathcal{V}^{s+1}(\R^2)}\big)>0$, a space $\mathcal{Y}^s_T$ continuously
embedded in $C([0, T]:H^s(\R^2)\times \mathcal{V}^{s+1}(\R^2))$ and a unique solution $(\eta,\Phi) \in \mathcal{Y}^s_T$ to the IVP associated to \eqref{ivp1}.
\end{Theorem}

Throughout the rest of this paper, we consider  $p=1$ in \eqref{eq1.3}.

\bigskip

\noindent
{\bf{Remarks}}:

\begin{enumerate}[i)]
\item Although we only treat the case $p=1$, the same method can be adapted for any integer value of $p>1$. The regularity for which LWP can be achieved will naturally depend on $p$.
\item This result not only extends the LWP claim in \cite{Q1} but it also provides a complete and detailed proof of the equivalence between the original system and the one obtained through differentiation, which seems to be missing there.
\end{enumerate}

\bigskip

\noindent
{\bf{Notation}}:
Finally we introduce some notation that we will use in this work. We use $\hat{f}$ to denote  the Fourier transform of $f$ and $H^s=(1-\Delta)^{- s/2}L^2 $  to denote the $L^2$-based Sobolev space of order $s$.
The Riesz potential of order $-s$ is denoted by $D_x^s =(-\Delta)^{s/2}$, so that one has $D_x^1 = \sqrt{-\Delta}$
with Fourier symbol given by $|\xi|$. We also define the Riesz transforms $R_l$, $l=1,2$ {\em via} the Fourier transform by $R_l(f)=\big(-i\frac{\xi_l}{|\xi|}\hat{f}\big)^{\vee}$.

For $f:\R^n\times [0, T] \to \R$ we define the mixed
 $L_T^q L_x^p$-norm by
\begin{equation*}
\|f\|_{L_T^q L_x^p} = \left(\int_0^T \left(\int_{\R} |f(x, t)|^p\,dx
\right)^{q/p}\,dt\right)^{1/q},
\end{equation*}
with usual modifications when $p = \infty$. An analogous definition is used for the other mixed norms $L_x^pL_T^q$,
with the order of integration in time and space interchanged.
We replace $T$ by $t$ if $[0, T]$ is the whole real line $\R$.

We use $c$ to denote various  constants whose exact values  may
 vary from one line to the next. We use $A\lesssim B$ to denote an estimate
of the form $A\leq cB$ for some $c$, and $A\sim B$ if $A\lesssim B$ and $B\lesssim A$. Also, we
use the notation $a+$ to denote $a+\epsilon$ for $\epsilon > 0$.


\section{Linear Estimates}
In this section we present some estimates satisfied by the group associated to the linear part of the diagonalized system \eqref{ivp2}. These estimates are very similar to the analogous ones in \cite{LPS}, and are actually simpler here due to the better symbol of the phase of the linear propagator.

We thus consider the following pair of linear equations
\begin{equation}\label{eq2.1}
\begin{cases}
w_t\pm i\sqrt{-\Delta}(1-\Delta)w = 0\\
w(x, 0) = w_0(x),
\end{cases}
\end{equation}
whose solution is described by the unitary groups $\widehat{U^{\mp}(t)w_0}(\xi) = e^{\mp it\phi(\xi)}\widehat{w_0}(\xi)$, where $\phi(\xi)$ is the Fourier symbol given by $\phi(\xi) =|\xi|^3+|\xi|$. Note that, as in \cite{LPS}, this symbol is a radial function, whose specific Fourier transform properties are a crucial ingredient exploited in the proofs that follow. We will only derive estimates associated to the unitary group $U:=U^+$, as the estimates for $U^-$ are the same.

We start with the smoothing estimate.
\begin{Theorem}\label{th.1}
Let $T>0$ and $\{Q_{\alpha}\}_{\alpha\in \Z^2}$ be a family of non-overlapping cubes of unit size so that $\R^2 = \cup_{\alpha\in \Z^2}Q_{\alpha}$. Then the following estimate holds
\begin{equation}\label{eq2.2}
\sup_{\alpha}\Big(\int_{Q_{\alpha}}\int_0^T |D_x^1U(t)w_0(x)|^2dtdx\Big)^{\frac12}\lesssim \|w_0\|_{L_x^2},
\end{equation}
where the implicit constant does not depend on T.
\end{Theorem}
\begin{proof}
Without loss of generality, we can consider the cubes $Q_{\alpha}= Q := \{x: |x|<1\}$. We have that the gradient of the Fourier symbol $|\nabla \phi(\xi)|=3|\xi|^2+1$ is always positive. Therefore from Theorem 4.1 in \cite{KPV1}, it follows that
\begin{equation}\label{eq2.3}
\|D_x^1U(t)w_0\|_{L^2_{Q_\times [0, T]}}\lesssim \Big(\int\frac{|\xi|^2}{3|\xi|^2+1}|\widehat{w_0}(\xi)|^2d\xi\Big)^{\frac12} \lesssim \|w_0\|_{L_x^2}.
\end{equation}
\end{proof}

Next we present the maximal function estimate. In the following theorem, the Fourier symbol will be $\phi(\xi)=|\xi|^3+|\xi|$, with $\xi\in \R^n$. Even though we only need the 2-dimensional case, this result can be established $n$ dimensions.

\begin{Theorem}\label{th.2}
Let $\{Q_{\alpha}\}_{\alpha\in \Z^n}$ be the mesh of dyadic cubes in $\R^n$. Then for any $s>\frac{3n}4$ and $T>0$, the following estimate holds
\begin{equation}\label{eq2.4}
\Big(\sum_{\alpha\in\Z^n}\sup_{|t|\leq T}\sup_{x\in Q_{\alpha}}|U(t)w_0(x)|^2\Big)^{\frac12}\lesssim (1+T^{\frac{n-1}4})\|w_0\|_{H^s}.
\end{equation}
\end{Theorem}

The proof of this theorem is totally analogous to the corresponding one in  \cite{LPS}, relying heavily on the particular representation of the Fourier transform of radial functions in terms of Bessel
functions. So we omit this
proof here, referring the reader to the fully detailed proof in that paper, observing only that we now need
to consider the phase $\phi(\xi)=|\xi|^3+|\xi|$ (instead of $\phi(\xi)=\epsilon |\xi|^3-|\xi|$ there).

As this representation of the Fourier transform of a radial function will also be required in the remaining part of this section, we gather some properties satisfied by it. The Fourier transform of a radial function $f(|x|)=f(s)$ is still a radial function  and is given by (see \cite{SW})
\begin{equation}\label{eq2.5}
\hat{f}(r) =\hat{f}(|\xi|) = r^{-\frac{n-2}2}\int_0^{\infty}f(s)J_{\frac{n-2}2}(rs)s^{\frac n2}ds,
\end{equation}
where $J_m$ is the Bessel function defined by
\begin{equation}\label{eq2.6}
J_m(r)=\frac{(r/2)^m}{\Gamma(m+1/2)\pi^{\frac12}}\int_{-1}^1e^{irs}(1-s^2)^{m-\frac12}ds, \qquad {\text{for}}\quad m>-\frac12.
\end{equation}
In the following lemma we state some properties of the Bessel function (see \cite{GPW}, \cite{SW} for details).
\begin{Lemma}\label{lema.1}
The Bessel function $J_m(r)$ satisfies the following properties.
\begin{equation}\label{eq2.7}
J_m(r) =O(r^m), \quad {\text{as}}\; r\to 0,
\end{equation}
\begin{equation}\label{eq2.8}
J_m(r) = e^{-ir}\sum_{j=0}^N\alpha_{m,j}r^{-(j+\frac12)}+e^{ir}\sum_{j=0}^N\tilde{\alpha}_{m,j}r^{-(j+\frac12)} +O\big(r^{-(N+\frac32)}\big), \quad {\text{as}}\; r\to +\infty,
\end{equation}
for any $N\in \Z_+$, and
\begin{equation}\label{eq2.9}
r^{-\frac{n-2}2}J_{\frac{n-2}2}(r) = c_n\mathcal{R}(e^{ir}h(r)),
\end{equation}
where $h$ is a smooth function satisfying
\begin{equation}\label{eq2.10}
|\partial_r^kh(r)|\leq c_k(1+r)^{-\frac{n-1}2-k},
\end{equation}
for any integer $k\geq 0$.
\end{Lemma}

Finally, our last goal is to derive a Strichartz estimate. As usual, we will obtain it from a $T T^*$ argument,
using an $L^1 - L^{\infty}$ decay estimate for the linear propagator together with the conservation of its $L^2$ norm.

We start by obtaining the decay estimate whose proof depends on the Van der Corput lemma.
\begin{Lemma}\label{van-der}
Let $f$ be a real valued $C^2$ function, and $\psi$ a $C^1$ function, defined in $[a, b]$, such that $|f^{''}(\xi)|>1$ for any $\xi\in [a, b]$. Then
\begin{equation}\label{van.2}
\Big|\int_a^be^{i\lambda f(\xi)}\psi(\xi)\, d\xi\Big| \lesssim |\lambda|^{-\frac12}\big(|\psi(b)|+\|\psi^{'}\|_{L^1([a,b])}\big),
\end{equation}
where the implicit constant is independent of $a$ or $b$.
\end{Lemma}

The crucial decay estimate now follows.
\begin{Lemma}\label{lema.4}
Let $0\leq \beta\leq 1$, then for all $t\in \R$ and $x\in \R^2$, one has
\begin{equation}\label{dec.1}
\|D_x^{\beta}U(t)w_0\|_{L_x^{\infty}}\lesssim t^{-\frac{2+\beta}3}\|w_0\|_{L^1}.
\end{equation}
\end{Lemma}
\begin{proof}
The idea of the proof of this lemma is very similar to the one used in \cite{LPS}. Using the Fourier transform and Young's inequality we have
\begin{equation}\label{eq2.25}
\|D_x^{\beta}U(t)w_0\|_{L_x^{\infty}}\leq \|(|\cdot|^{\beta}e^{it\phi})^{\vee}\|_{L_x^{\infty}}\|w_0\|_{L^1}.
\end{equation}
The formula for the Fourier transform of a radial function \eqref{eq2.5} and the property \eqref{eq2.9} yield
\begin{equation}\label{eq2.26}
\begin{split}
(|\xi|^{\beta}e^{it\phi(\xi)})^{\vee}(x) &=\int_0^{+\infty}s^{\beta}e^{it(s^3+s)}J_0(rs)s\,ds\\
&=\int_0^{+\infty}s^{\beta+1}e^{it(s^3+s)}e^{irs}h(rs)ds +\int_0^{+\infty}s^{\beta+1}e^{it(s^3+s)}e^{-irs}\overline{h(rs)}ds\\
&=A_{\beta}(r,t) +B_{\beta}(r,t),
\end{split}
\end{equation}
where $r=|x|$. Without loss of generality we consider $t\geq 0$. The estimates for both terms in \eqref{eq2.26} follow the same techniques and the
$A_{\beta}(r,t)$ case is actually simpler, so we consider only the second term, $B_{\beta}(r,t)$. Performing the change of variables $u =t^{\frac13}s$, we have
\begin{equation}\label{eq2.27}
\begin{split}
B_{\beta}(r,t) &= t^{-\frac{2+\beta}3}\int_0^{+\infty}u^{\beta+1}e^{i(u^3+t^{\frac23}u-r t^{-\frac13}u)}\overline{h(rt^{-\frac13}u)}du\\
&=:t^{-\frac{2+\beta}3}M_{\beta}(rt^{-\frac13}),
\end{split}
\end{equation}
where the function $M_{\beta}(r)$ is defined as
\begin{equation*}
M_{\beta}(r) =\int_0^{+\infty}u^{\beta+1}e^{if_r(u)}\overline{h(ru)}du,
\end{equation*}
with the phase function, for $\gamma = t^{\frac23}$, given by
\begin{equation*}
f_r(u) = u^3-(r-\gamma)u.
\end{equation*}

The proof of the lemma then reduces to proving that
\begin{equation}\label{eq2.28}
\sup_{r\geq 0}|M_{\beta}(r)|\lesssim 1, \qquad \; \forall\; \beta\in[0, 1].
\end{equation}

We now use standard oscillatory integral methods to obtain \eqref{eq2.28}. By the stationary phase principle, the main contribution to the integral comes from
the points where the derivative of the phase vanishes, i.e.
$$
3u^2 = r-\gamma.
$$
So we isolate the integral around these points, with a smooth cut-off function $\psi:\R \to \R$, $0\leq\psi\leq 1$, rescaled around
$u=\sqrt{\frac{r-\gamma}{3}}$, such that
$$
\mbox{supp}\, \psi \subset \{ u : |3u^2-(r-\gamma)| \leq \frac{r-\gamma}{2} \},
$$
and
$$\psi(u)=1 \qquad \mbox{for} \qquad |3u^2-(r-\gamma)| \leq \frac{r-\gamma}{4}.$$
Notice, as this will be important below, that the length of the support is of the order $(r-\gamma)^{1/2}$ and that each derivative of $\psi$ produces a factor
of $(r-\gamma)^{-1/2}$.
Splitting the integral $M_{\beta}(r)$ into two terms, with this cut-off function, we obtain
$$
|M_{\beta}(r)| \leq |M^1_{\beta}(r)|+|M^2_{\beta}(r)|,
$$
where
$$
M^1_{\beta}(r)=\int_0^{+\infty}u^{\beta+1}e^{if_r(u)}\overline{h(ru)} \,\psi(u) du,
$$
and
$$
M^2_{\beta}(r)=\int_0^{+\infty}u^{\beta+1}e^{if_r(u)}\overline{h(ru)} (1-\psi(u)) du.
$$
When $r-\gamma \leq 0$, no stationary phase points exist and no splitting is required, with the whole integral reducing
to the $M^2_{\beta}(r)$ case. This is, in fact, what happens for the $A_\beta(r,t)$ term above, whose phase is given by $f_r(u) = u^3+(r+\gamma)u$ and therefore
has no stationary phase points in the domain of integration $u \in (0,\infty)$.

The bound for the non-stationary phase term $M^2_{\beta}(r)$ is easily obtained integrating by parts in the usual way, using the non-vanishing derivative of the phase (see \cite{LPS} for details).

As for $M^1_{\beta}(r)$, we now use Van der Corput's lemma to obtain an accurate bound for the integral around the stationary phase points. We start by
writing $f_r(u)=(r-\gamma)^{1/2}\Big( \frac{u^3}{(r-\gamma)^{1/2}}-(r-\gamma)^{1/2}u\Big)$ and note that on the support of $\psi$ we have $u \sim (r-\gamma)^{1/2}$
so that
$$
\Big( \frac{u^3}{(r-\gamma)^{1/2}}-(r-\gamma)^{1/2}u\Big)''=\frac{6u}{(r-\gamma)^{1/2}} \gtrsim 1,
$$
therefore
\begin{equation}
\label{vdc}
|M^1_{\beta}(r)| \lesssim  \frac{1}{(r-\gamma)^{1/4}} \Big(\|h(r \cdot) (\cdot)^{\beta +1} \psi(\cdot)\|_{L^{\infty}} +
\|\frac{d}{du}\big(h(r \cdot) (\cdot)^{\beta +1} \psi(\cdot)\big)\|_{L^1}
\Big).
\end{equation}
Using the estimate for $h$ and its derivatives \eqref{eq2.10}, the fact that on the support of $\psi$ we have $u \sim (r-\gamma)^{1/2}$, and the observations
made earlier about the derivatives of $\psi$ and the length of its support, it is easy to check that all the terms on the RHS of the inequality \eqref{vdc} are
equally bounded by a constant multiple of
$$
\frac{(r-\gamma)^{(\beta+1)/2}}{(r-\gamma)^{1/4}(1+r(r-\gamma)^{1/2})^{1/2}},
$$
yielding
$$|M^1_{\beta}(r)|\lesssim 1,$$
with a constant depending on $\beta$.
\end{proof}

Having obtained the decay estimate, we can now proceed to use it as one of the main ingredients in the proof of the Strichartz estimate for the unitary
linear propagator.

\begin{Theorem}\label{th.3}
Let $T>0$ and $0\leq \delta <\frac12$. Then, for all $w_0\in L^2(\R^2)$
\begin{equation}\label{eq2.23}
\|D_x^{\delta}U w_0\|_{L_T^{q_{\delta}}L_x^{\infty}}\lesssim \|w_0\|_{L^2},
\end{equation}
where $q_{\delta} = \frac3{1+\delta}$.
\end{Theorem}

\begin{proof} Proofs of these types of estimates, as consequence of the $L^2$ conservation and the $L^1 - L^{\infty}$ decay estimates for the corresponding linear
evolution groups, are now a standard procedure (see \cite{KT} or \cite{KPV1} for example).

We start by interpolating
$$\|U(t)w_0\|_{L_x^2} = \|w_0\|_{L^2},$$
and
$$\|D_x^{\beta}U(t)w_0\|_{L_x^{\infty}}\lesssim t^{-\frac{2+\beta}3}\|w_0\|_{L^1},$$
to obtain
$$\|D_x^{\beta \theta}U(t)w_0\|_{L_x^{p}}\lesssim t^{-\frac{2+\beta}3 \theta}\|w_0\|_{L^{p'}},$$
for $0\leq \theta \leq 1$, $\frac1p = \frac12 - \frac{\theta}{2}$ and $1=\frac 1p+ \frac{1}{p'}$.

Then, using a $T T^*$ argument on the dual formulation and the Hardy-Littlewood-Sobolev inequality in the $t$ variable, we get
$$\|D_x^{\beta \theta/2}U(t)w_0\|_{L^q_t L_x^{p}}\lesssim \|w_0\|_{L^2},$$
for $0\leq \theta \leq 1$, $\frac1p = \frac12 - \frac{\theta}{2}$, $\frac2q=\frac{2+\beta}{3}\theta$.

Finally, choosing the endpoint case $\theta=1$ and writing $\delta=\beta/2$ we obtain the desired inequality. Notice that, at $\theta=1$, the case $\beta=1$, i.e. $\delta=\frac12$, becomes unattainable, due to the Hardy-Littlewood-Sobolev restriction $\frac{2+\beta}{3}\theta <1$. This corresponds to the known $L^2_t L_x^{\infty}$ Strichartz
endpoint problem.
\end{proof}

Before leaving this section we record two more results that will be required in the sequel. The first one is the Leibniz's rule for fractional derivatives whose proof can be found in \cite{KPV4}.
\begin{Lemma}\label{leibniz}
Let $0<m<1$ and $1<p<\infty$. Then for any $f,\, g:\R^n\to \C$, one has
\begin{equation}\label{leb.1}
\|D_x^m(fg)-fD_x^mg-gD_x^mf\|_{L_x^p}\lesssim \|g\|_{L_x^{\infty}}\|D_x^mf\|_{L_x^p}.
\end{equation}
\end{Lemma}
The second result is a well known theorem in harmonic analysis, about the $L^p$ boundedness of the Riesz transform, which states the following.
\begin{Lemma}\label{lem.rz}
Let $R_l$, $l=1,2$ be the Riesz transform. Then, for $1<p<\infty$
\begin{equation}\label{rz.1}
\|R_lf\|_{L_x^p} \lesssim \|f\|_{L_x^p}.
\end{equation}
\end{Lemma}


\section{Proof of the local well-posedness result}

In this section we provide a proof of the main result of this work. As explained in the introduction, we start by proving a local well-posedness result for the $4\times 4$ system that comes by differentiating each equation in \eqref{ivp2} with respect to ${x_1}$ and ${x_2}$. Introducing the new unknowns $w_1:=u_{x_1}$, $w_2:=u_{x_2}$, $w_3:=v_{x_1}$and $w_4:=v_{x_2}$, we obtain the following IVP,
\begin{equation}\label{eq3.1}
\begin{cases}
\partial_tw_1 + i\sqrt{-\Delta}(1-\Delta)w_1 = \frac12\big(R_1\partial_{x_1}L_1
+ R_2\partial_{x_1}L_2\big)-\frac14\partial_{x_1}L_3\\
\partial_tw_2 + i\sqrt{-\Delta}(1-\Delta)w_2 =  \frac12\big(R_1\partial_{x_2}L_1
+ R_2\partial_{x_2}L_2\big)-\frac14\partial_{x_2}L_3\\
\partial_tw_1 - i\sqrt{-\Delta}(1-\Delta)w_1 = -\frac12\big(R_1\partial_{x_1}L_1
+ R_2\partial_{x_1}L_2\big)-\frac14\partial_{x_1}L_3\\
\partial_tw_2 - i\sqrt{-\Delta}(1-\Delta)w_2 =  -\frac12\big(R_1\partial_{x_2}L_1
+ R_2\partial_{x_2}L_2\big)-\frac14\partial_{x_2}L_3\\
w_1(x_1, x_2, 0) = w_1(0),\quad  w_2(x_1, x_2, 0) = w_2(0),\\
w_3(x_1, x_2, 0) = w_3(0),\quad  w_4(x_1, x_2, 0) = w_4(0)
\end{cases}
\end{equation}
where
\begin{equation}\label{eq3.2}
\begin{split}
L_1&:= L_1(w_1,w_2,w_3,w_4)=(R_1w_1+R_2w_2-R_1w_3-R_2w_4)(w_1+w_3)\\
L_2&:= L_2(w_1,w_2,w_3,w_4)=(R_1w_1+R_2w_2-R_1w_3-R_2w_4)(w_2+w_4)\\
L_3&:= L_3(w_1,w_2,w_3,w_4)=(w_1+w_3)^2 + (w_2+w_4)^2.
\end{split}
\end{equation}

In what follows we give a local well-posedness result for the IVP \eqref{eq3.1}. Before presenting this result, we introduce some notations. Define  $\mathcal{H}^s(\R^2):= \big(H^s(\R^2)\big)^4$ with norm given by $\|g\|_{\mathcal{H}^s(\R^2)}:=\sum_{j=1}^4\|g_j\|_{H^s(\R^2)}$, for $g=(g_1, g_2, g_3, g_4) \in
\mathcal{H}^s(\R^2)$. We also denote by $w=(w_1, w_2, w_3, w_4)$ the space-time functions in $(C([0, T]:H ^s(\R^2)))^4 \cong C([0, T]:\mathcal{H}^s(\R^2))$.

\begin{Theorem}\label{loc-b}
Let $s>3/2$, then for any $w(0)\in \mathcal{H}^s(\R^2)$, there exist a time $T=T\big(\|w(0)\|_{\mathcal{H}^s(\R^2)}\big)$, a space $\mathcal{X}^s_T$ continuously
embedded in $C([0, T]:\mathcal{H}^s(\R^2))$ and a unique solution $w \in \mathcal{X}^s_T$ to the Cauchy problem \eqref{eq3.1}.
\end{Theorem}

\begin{proof} We consider only the most difficult cases $\frac32< s<2$, as for $s\geq 2$ the proof is slightly simpler. Let the fractional part of $s$ be $m:=s-1$ so that $m\in (1/2, 1)$, and let $\delta=1-m=2-s$, defining $q_{\delta}$ as in Theorem \ref{th.3}. For $f \in C([0, T]:H ^s(\R^2))$ define

\begin{equation}\label{eq3.3}
\Omega_T^1(f):=\sup_{0\leq t\leq T}\|f\|_{H^s},
\end{equation}
\begin{equation}\label{eq3.4}
\Omega_T^2(f):= \sum_{|\beta|\leq 1}\|\partial_x^{\beta}f\|_{L_T^3L_x^{\infty}}+\|D_x^1f\|_{L_T^3L_x^{\infty}} +\sum_{l=1,2}\sum_{|\beta|\leq 1}\|\partial_x^{\beta}R_lf\|_{L_T^3L_x^{\infty}} +\sum_{l=1,2}\|D_x^1R_lf\|_{L_T^3L_x^{\infty}},
\end{equation}

\begin{equation}\label{eq3.5}
\Omega_T^3(f):= \sum_{|\beta|= 1}\|D_x^1\partial_x^{\beta}f\|_{L_T^{q_\delta}L_x^{\infty}} +\sum_{l=1,2}\sum_{|\beta|= 1}\|D_x^1\partial_x^{\beta}R_lf\|_{L_T^{q_\delta}L_x^{\infty}},
\end{equation}

\begin{equation}\label{eq3.6}
\Omega_T^4(f):= \sum_{|\beta|= 1}\sup_{\alpha\in\Z^2}\|\partial_x^{\beta}D_x^{1+m}f\|_{L_{Q_\alpha\times [0, T]}^2} +\sum_{l=1,2}\sum_{|\beta|= 1}\sup_{\alpha\in\Z^2}\|\partial_x^{\beta}D_x^{1+m}R_lf\|_{L_{Q_\alpha\times [0, T]}^2},
\end{equation}
\begin{equation}\label{eq3.7}
\Omega_T^5(f):= \Big(\sum_{\alpha\in\Z^2}\|f\|^2_{L_{Q_\alpha\times [0, T]}^{\infty}}\Big)^{\frac12}
+\sum_{l=1,2}\Big(\sum_{\alpha\in\Z^2}\|R_lf\|^2_{L_{Q_\alpha\times [0, T]}^{\infty}}\Big)^{\frac12},
\end{equation}
and let $\Omega_T(f):=\max_{j=1,\cdots,5}\Omega_T^j(f)$.

Define the Banach space $X_T^s$ by
\begin{equation*}
X_T^s:= \{f\in C([0, T]:H ^s(\R^2)): \Omega_T(f)<\infty\},
\end{equation*}
with norm $\|f\|_{X_T^s}:=\Omega_T(f)$, and finally define $\mathcal{X}_T^s:=(X_T^s)^4$ with norm $\|w\|_{\mathcal{X}_T^s}:=\sum_{j=1}^4 \|w_j\|_{X_T^s}$, identifying
$(C([0, T]:H ^s(\R^2)))^4$ with $C([0, T]:\mathcal{H}^s(\R^2))$.

For the given initial data $w(0)=(w_1(0), w_2(0), w_3(0), w_4(0)) \in \mathcal{H}^s(\R^2)$ and using Duhamel's formula, we define applications
 \begin{equation}\label{eq3.8}
 \begin{cases}
 \Psi_1(w_1):= U^+(t)(w_1(0)) +\frac12\int_0^t U^+(t-t')\big(R_1\partial_{x_1}L_1
+ R_2\partial_{x_1}L_2-\frac14\partial_{x_1}L_3\big)(t')dt'\\
\Psi_2(w_2):= U^+(t)(w_2(0)) +\frac12\int_0^t U^+(t-t')\big(R_1\partial_{x_2}L_1
+ R_2\partial_{x_2}L_2-\frac14\partial_{x_2}L_3\big)(t')dt'\\
\Psi_3(w_3):= U^-(t)(w_3(0)) -\frac12\int_0^t U^-(t-t') \big(R_1\partial_{x_1}L_1
+ R_2\partial_{x_1}L_2+\frac14\partial_{x_1}L_3\big)(t')dt'\\
\Psi_4(w_4):= U^-(t)(w_4(0)) -\frac12\int_0^t U^-(t-t') \big(R_1\partial_{x_2}L_1
+ R_2\partial_{x_2}L_2+\frac14\partial_{x_2}L_3\big)(t')dt'.
\end{cases}
\end{equation}

Define a ball with center at the origin and radius $a>0$ in $\mathcal{X}_T^s$ by
$\mathcal{B}_a:= \{f\in \mathcal{X}_T^s: \|f\|_{\mathcal{X}_T^s}\leq a\}$. We will prove that there exist $a>0$ and $T>0$ such that the application $\Psi: =(\Psi_1, \Psi_2, \Psi_3, \Psi_4)$ is a contraction map on the ball $\mathcal{B}_a$.

Using linear estimates \eqref{eq2.2}, \eqref{eq2.4}, \eqref{eq2.23}, we have that
\begin{equation}\label{eq3.9}
\|\Psi_j(w_j)\|_{X_T^s}\leq c_0\|w_j(0)\|_{H^s(\R^2)}+c\int_0^T\|F_j(w_1, w_2, w_3, w_4)(t')\|_{H^s(\R^2)}dt',
\end{equation}
 for all $j=1,\cdots, 4$ where $F_j(w_1, w_2, w_3, w_4)$ are the respective nonlinear parts in equation \eqref{eq3.1}.

 The estimates for each component $\Psi_j(w_j)$ follow similar techniques, therefore we present details only for the nonlinear part associated to the first component, $\Psi_1(w_1)$. Note that in view of \eqref{eq3.1} and \eqref{eq3.2} the terms in $F_1(w)$ are of the form  $R_i\partial_{x_1}\{(R_lw_j)w_k\}$ and $\partial_{x_1}(w_jw_k)$, for $j, k =1, \cdots, 4$ and $i,l=1,2$.

 Again, we only provide details for the terms that are in  the form $R_i\partial_{x_1}\{(R_lw_j)w_k\}$,  i.e., we find estimates for $\int_0^T\|R_i\partial_{x_1}\{(R_lw_j)w_k\}(t')\|_{H_x^s}dt'$. The estimates for the other terms follow a similar method and are even easier to handle. Using the fact that the Riesz transform is bounded in $L^p$, $1<p<\infty$, we have
 \begin{equation}\label{eq3.10}
 \begin{split}
 \int_0^T\|R_i\partial_{x_1}\{(R_lw_j)w_k\}(t')\|_{H_x^s}dt'&\leq  \int_0^T\|\partial_{x_1}\{(R_lw_j)w_k\}(t')\|_{L_x^2}dt'\\
 &\qquad + \int_0^T\|D_x^s\partial_{x_1}\{(R_lw_j)w_k\}(t')\|_{L_x^2}dt'\\
 &=:I+II.
 \end{split}
 \end{equation}

Now,
 \begin{equation}\label{eq3.11}
 \begin{split}
 I&\lesssim \int_0^T\|(R_l\partial_{x_1}w_j) w_k\|_{L_x^2}dt'+\int_0^T\|(R_lw_j) (\partial_{x_1}w_k)\|_{L_x^2}dt'\\
 &\lesssim \int_0^T\|R_l\partial_{x_1}w_j\|_{L_x^2}\|w_k\|_{L_x^{\infty}}dt' + \int_0^T\|R_lw_j\|_{L_x^2} \|\partial_{x_1}w_k\|_{L_x^{\infty}}dt'\\
 &\lesssim \|w_k\|_{L_T^{\infty}L_x^{\infty}} T\|\partial_{x_1}w_j\|_{L_T^{\infty}L_x^2} +\|w_j\|_{L_T^{\infty}L_x^2}T^{\frac32}\|\partial_{x_1}w_k\|_{L_T^3L_x^{\infty}}\\
 &\lesssim T\Omega_T^1(w_k) \Omega_T^1(w_j) +T^{\frac32}\Omega_T^1(w_j) \Omega_T^2(w_k).
 \end{split}
 \end{equation}

 To estimate $II$, recall the notation $\sqrt{-\Delta}=D_x^1 = R_1\partial_{x_1}+R_2\partial_{x_2}$ so that we can write $D_x^s = D_x^1D_x^m = D_x^mR_1\partial_{x_1}+D_x^mR_2\partial_{x_2}$. Therefore we have
 $$D_x^s\partial_{x_1}\{(R_lw_j)w_k\}=D_x^m\big( (R_l\partial_{x_1}D_x^1w_j )w_k +(R_l\partial_{x_1}w_j)D_x^1w_k + (R_lD_x^1w_j)\partial_{x_1}w_k+(R_lw_j)D_x^1\partial_{x_1}w_k\big).$$
 With this decomposition, we can split $II$ into four terms
 \begin{equation}\label{eq3.12}
 \begin{split}
 II&\leq \int_0^T\|D_x^m\big( R_l\partial_{x_1}D_x^1w_j w_k\big)\|_{L_x^2}dt'
 +\int_0^T\|D_x^m\big(R_l\partial_{x_1}w_jD_x^1w_k \big)\|_{L_x^2}dt'\\
 &\qquad +\int_0^T\|D_x^m\big(R_lD_x^1w_j\partial_{x_1}w_k \big)\|_{L_x^2}dt'
 +\int_0^T\|D_x^m\big( R_lw_jD_x^1\partial_{x_1}w_k\big)\|_{L_x^2}dt'\\
 &=: a+b+c+d.
 \end{split}
 \end{equation}

 Using Leibniz's rule for fractional derivatives \eqref{leb.1}, $L^p$-boundedness of the Riesz transform \eqref{rz.1} and H\"older's inequality, one gets
 \begin{equation}\label{eq3.13}
 \begin{split}
 a&\lesssim \int_0^T\|R_l\partial_{x_1}D_x^1w_j\|_{L_x^{\infty}}\|D_x^mw_k\|_{L_x^2}dt' +\int_0^T\|(R_l\partial_{x_1}D_x^{1+m}w_j) w_k\|_{L_x^2}dt'\\
 &\lesssim \|D_x^mw_k\|_{L_T^{\infty}L_x^2}T^{\frac1{q_{\delta}'}}\|R_l\partial_{x_1}D_x^1w_j\|_{L_T^{q_{\delta}}L_x^{\infty}} + T^{\frac12}\Big(\sum_{\alpha\in\Z^2}\!\int_0^T\!\!\!\!\int_{Q_{\alpha}}\!\!\!\!|w_kR_l\partial_{x_1}D_x^{1+m}w_j|^2dxdt\Big)^{\frac12}\\
 &\lesssim T^{\frac1{q_{\delta}'}}\Omega_T^1(w_k) \Omega_T^3(w_j)
  +T^{\frac12}\Big(\sum_{\alpha\in\Z^2}\|w_k\|_{L^{\infty}_{Q_\alpha\times [0, T]}}^2\Big)^{\frac12} \sup_{\alpha\in\Z^2}\|R_l\partial_{x_1}D_x^{1+m}w_j\|_{L_{Q_\alpha\times [0, T]}^2}\\
 &\lesssim T^{\frac1{q_{\delta}'}}\Omega_T^1(w_k) \Omega_T^3(w_j)+T^{\frac12}\Omega_T^5(w_k)\Omega_T^4(w_j).
 \end{split}
 \end{equation}

 Likewise
 \begin{equation}\label{eq3.14}
 \begin{split}
 b&\lesssim \int_0^T\|\partial_{x_1}R_lw_j\|_{L_x^{\infty}}\|D_x^{1+m}w_k\|_{L_x^2}dt' +\int_0^T\|(D_x^m\partial_{x_1}R_lw_j) D_x^1w_k\|_{L_x^2}dt'\\
 &\lesssim \|D_x^{1+m}w_k\|_{L_T^{\infty}L_x^2}\int_0^T\|\partial_{x_1}R_lw_j\|_{L_x^{\infty}}dt' +\int_0^T\|D_x^1w_k\|_{L_x^{\infty}}\|D_x^m\partial_{x_1}R_lw_j\|_{L_x^2}dt'\\
 &\lesssim\Omega_T^1(w_k)T^{\frac32}\|\partial_{x_1}R_lw_j\|_{L_T^3L_x^{\infty}} + \|D_x^m\partial_{x_1}w_j\|_{L_T^{\infty}L_x^2}T^{\frac32}\|D_x^1w_k\|_{L_T^3L_x^{\infty}}\\
 & \lesssim T^{\frac32}\Omega_T^1(w_k)\Omega_T^2(w_j) + T^{\frac32}\Omega_T^1(w_j)\Omega_T^2(w_k).
 \end{split}
 \end{equation}

The estimates for the terms $c$ and $d$ are similar to those for the terms $b$ and $a$, respectively, so we omit the details. Inserting the estimates \eqref{eq3.10}--\eqref{eq3.14} in \eqref{eq3.9}, one obtains
\begin{equation}\label{eq3.15}
\|\Psi_1(w_1)\|_{X_T^s}\leq c_0\|w_1(0)\|_{H^s(\R^2)}+cT^{\varepsilon}\big(\|w\|_{\mathcal{X}_T^s}\big)^2,
\end{equation}
for some $\varepsilon>0$. Gathering the similar estimates for the other components of $\Psi$, we get
\begin{equation}\label{eq3.16}
 \|\Psi(w)\|_{\mathcal{X}_T^s}\leq c_0\|w(0)\|_{\mathcal{H}^s(\R^2)}+4cT^{\varepsilon}\big(\|w\|_{\mathcal{X}_T^s}\big)^2.
 \end{equation}

 Given $w(0)\in\mathcal{H}^s(\R^2)$, if we choose $a= 2c_0\|w(0)\|_{\mathcal{H}^s(\R^2)}$ and $T>0$ in such a way that $4cT^{\varepsilon}\|w\|_{\mathcal{X}_T^s}<\frac12$, it is easy to show that $\Psi$ maps $\mathcal{B}_a$ onto itself. A totally analogous method is then used to prove that the application $\Psi$ is a contraction mapping. The remaining part of the proof follows by a standard argument.
\end{proof}

\begin{proof}[Proof of Theorem \ref{mainTh}] Having established Theorem \ref{loc-b}, the proof of Theorem \ref{mainTh} follows easily. For data $\eta_0 \in H^s(\R^2)$ and $\Phi_0 \in \mathcal{V}^{s+1}(\R^2)$, we start by defining
\begin{equation}
\label{di}
u_0=-\frac{i}{2\sqrt{-\Delta}}\eta_0 + \frac12 \Phi_0 \qquad \mbox{and} \qquad v_0=\frac{i}{2\sqrt{-\Delta}}\eta_0 + \frac12 \Phi_0,
\end{equation}
and then
$$w_1(0)=\pd_{x_1}u_0, \quad w_2(0)=\pd_{x_2}u_0, \quad w_3(0)=\pd_{x_1}v_0, \quad w_4(0)=\pd_{x_2}v_0,$$
noting that, from our definition of $\mathcal{V}^{s+1}(\R^2)$ we get $w(0) \in \mathcal{H}^s(\R^2)$, with $s>3/2$.

We can now apply Theorem \ref{loc-b} to $w(0)$, obtaining a unique solution $w(t) \in \mathcal{X}^s_T$. Differentiating
the first equation in \eqref{eq3.1} in $x_2$ and the second in $x_1$, shows that
$$\pd_t (\pd_{x_2}w_1-\pd_{x_1}w_2) =0.$$
Doing the same for the third and fourth equations, equally yields
$$\pd_t (\pd_{x_2}w_3-\pd_{x_1}w_4) =0.$$
From our construction of the initial data, we have
$$\pd_{x_2}w_1(0)-\pd_{x_1}w_2(0)=0 \qquad \mbox{and} \qquad \pd_{x_2}w_3(0)-\pd_{x_1}w_4(0)=0,$$
which, from the two previous relations, proves that the vector fields $(w_1,w_2)$ and $(w_3,w_4)$ are closed in $\R^2$ for all $t \in [0,T]$. Therefore,
there exist unique $u,v \in C([0,T] : \mathcal{V}^{s+1}(\R^2))$ such that $(\pd_{x_1} u, \pd_{x_2} u)=(w_1(t), w_2(t))$ and $(\pd_{x_1} v, \pd_{x_2} v)
=(w_3(t), w_4(t))$, satisfying $u(0)=u_0$ and $v(0)=v_0$ according to \eqref{di}. Finally, we invert back from $(u,v)$ to $(\eta, \Phi)$ using formula \eqref{eq1.6}. Tracing the steps of this proof, we observe that this solution is obtained uniquely in the space
$\mathcal{Y}_T^s=\{(\eta,\Phi) : \eta \in X^s_T, \,\, \sqrt{-\Delta} \Phi \in X^s_T\}$ which is continuously embedded in $C([0, T]:H^s(\R^2)\times \mathcal{V}^{s+1}(\R^2)) \cong C([0, T]:H^s(\R^2))\times C([0, T]:\mathcal{V}^{s+1}(\R^2))$.
\end{proof}


\end{document}